\documentclass{amsart}
\usepackage{amsfonts}
\usepackage{amsmath}
\usepackage{mathrsfs}
\usepackage{fullpage}
\usepackage{amssymb}
\usepackage{graphicx}
\newtheorem{theorem}{Theorem}
\newtheorem{lemma}[theorem]{Lemma}
\newtheorem{proposition}[theorem]{Proposition}
\newtheorem{corollary}[theorem]{Corollary}
\theoremstyle{definition}

\newtheorem{definition}[theorem]{Definition}
\newtheorem{example}[theorem]{Example}

\newtheorem{problem}[theorem]{Problem}
\newtheorem{remark}[theorem]{Remark}

\numberwithin{equation}{section}

\newcommand{\seq}[1]{\left \{#1 \right \}_{k=0}^{\infty}}
\newcommand{\R}{\mathbb{R}}

\newcommand{\LP}{\mathscr{L}-\mathscr{P}}
\newcommand{\ds}[1]{\displaystyle{#1}}

\title{Multiplier sequences,  classes of generalized Bessel functions and open problems}

\author{George Csordas and Tam\'as Forg\'acs}
\email{george@math.hawaii.edu}
\email{tforgacs@csufresno.edu}
\begin{document}

\maketitle
\begin{abstract} Motivated by the study of the distribution of zeros of generalized Bessel-type functions, the principal goal of this paper is to identify new research directions in the theory of multiplier sequences. The investigations focus on multiplier sequences interpolated by functions which are not entire and sums, averages and parametrized families of multiplier sequences. The main results include (i) the development of  a `logarithmic' multiplier sequence and  (ii)  several integral representations of a generalized Bessel-type  function utilizing some ideas of G. H. Hardy and  L. V. Ostrovskii. The  explorations and analysis,  augmented throughout  the paper by a plethora  of examples, led to a number of conjectures and intriguing open problems.

\end{abstract}
\section{Introduction}

In 1905 G.~H.~Hardy \cite{hardy_paper} studied the following entire functions of exponential type, as generalizations of $e^z$:
\begin{equation}\label{hardy}
E_{s,a}(z):=\sum_{n=0}^{\infty} (n+a)^s \frac{z^n}{n!}, \qquad s \in \mathbb{R}, \quad a \geq 0.
\end{equation}
Although Hardy allowed the parameters to be complex numbers, in the present paper we will only consider parameters satisfying the restrictions in (\ref{hardy}). Note that $E_{0,a}=e^z$, and for $k \in \mathbb{N}$, $E_{k,a}=e^zT_k(z)$, where $T_k(z)$ is a polynomial of degree $k$. If $a=0$, we set $\displaystyle{E_{s,0}=\sum_{n=1}^{\infty} n^s \frac{x^n}{n!}}$, $s \in \mathbb{R}$. In \cite{ostrovskii}, I.~V.~Ostrovskii describes the real zeros of these generalized exponential functions.
\begin{theorem}\label{ostrovthm} (\cite[Theorem 2.5]{ostrovskii}) Let $E_{s,a}$ be defined as in (\ref{hardy}), and let $k \in \mathbb{N}_0$. 
\begin{itemize}
\item[(a)] For $k<s<k+1$, $E_{s,a}$ has only $k+1$ real zeros.
\item[(b)] For $s<0$, $E_{s,a}$ does not have any real zeros.
\end{itemize}
\end{theorem} 
A consequence of Theorem \ref{ostrovthm} is that $\seq{(k+a)^s}$ is not a multiplier sequence (cf. Definition \ref{CMS}) for non-integral or negative $s$. The observation that the sequence $\seq{1/k!}$ is a complex zero decreasing sequence (cf. Definition \ref{CZDS}), however, motivates the study of functions of the form 
\begin{equation} \label{genBessel}
B_{s,a}(z):=\sum_{n=0}^{\infty} (n+a)^s \frac{z^n}{n! n!}
\end{equation}
along with the location of their zeros (see Section \ref{generalizedbessel}). We close this section with some definitions, and the general question which led to most of the work and considerations in this paper. \\
\begin{definition}\label{LP}  A real entire function $\displaystyle{\varphi(x)=\sum_{k=0}^{\infty} \frac{\gamma_k}{k!}x^k}$ is said to belong to the {\it Laguerre-P\'olya class}, written $\varphi \in \LP$, if it admits the representation
\[
\varphi(x)=c x^m e^{-ax^2+bx} \prod_{k=1}^{\omega} \left(1+\frac{x}{x_k} \right) e^{-x/x_k},
\]
where $b,c \in \R$, $x_k \in \R \setminus \{ 0\}$, $m$ is a non-negative integer, $a\geq 0$, $0 \leq \omega \leq \infty$ and $\displaystyle{\sum_{k=1}^{\omega} \frac{1}{x_k^2} < +\infty}$.
\end{definition}
\begin{definition}\label{LPI} A real entire function $\displaystyle{\varphi(x)=\sum_{k=0}^{\infty} \frac{\gamma_k}{k!}x^k}$ is said to be of {\it type I} in the Laguerre-P\'olya class, written $\varphi \in \LP I$, if $\varphi(x)$ or $\varphi(-x)$ admits the representation
\[
\varphi(x)=c x^m e^{\sigma x} \prod_{k=1}^{\omega} \left(1+\frac{x}{x_k} \right),
\]
where $c \in \R$, $m$ is a non-negative integer, $\sigma \geq 0$, $x_k >0$, $0 \leq \omega \leq \infty$ and $\displaystyle{\sum_{k=1}^{\omega} \frac{1}{x_k} < +\infty}$. If $\gamma_k \geq 0$ for $k=0,1,2.\ldots$, we write $\varphi \in \LP^+$. Finally, $\LP(-\infty,0]$ denotes the class of functions in $\LP$ whose zeros lie in $(-\infty,0]$.
\end{definition}
We point out that a real entire function $\varphi$ belongs to $\LP I$ if and only if its Taylor coefficients  are of the same sign, or alternate in sign. Thus $\LP^+\subset \LP I \subset \LP$.
\begin{definition}\label{CZDS} A sequence of real numbers $\seq{\gamma_k}$ is called a {\it complex zero decreasing sequence, or CZDS}, if the linear operator $T$ defined by $T[x^k]=\gamma_k x^k$ has the property that for every real polynomial $p(x)$, 
\[
Z_C(T[p(x)]) \leq Z_C(p(x)),
\]
where $Z_C(p)$ denotes the number of non-real zeros of the polynomial $p$, counting multiplicity.
\end{definition}
\begin{definition}\label{CMS} A sequence of real numbers $\seq{\gamma_k}$ is called a (classical\footnote{The original nomenclature for such sequences did not include the adjective {\it classical}. Indeed, P\'olya and Schur in \cite{PS} called these simply {\it Faktorenfolgen erster Art}. More recently, research has focused on sequences giving rise to linear operators that are diagonal with respect to a basis other than the standard one, necessitating the introduction of modifiers. We now talk about Hermite-, Laguerre-, Legendre- and Chebyshev-multiplier sequences (see for example \cite{BC}, \cite{tom_hermite}, \cite{bo}, \cite{tom_laguerre}, \cite{tom_legendre}, \cite{yoshi}). Consequently, we use the word `classical' to describe multiplier sequences whose operators are diagonal with respect to the standard basis.}) {\it multiplier sequence} (of the first kind), if the associated linear operator $T$ defined by $T[x^k]=\gamma_k x^k$, for $k=0,1,2,\ldots$, has the property that for every real polynomial $p(x)$, 
\[
Z_C(T[p(x)])=0 \qquad \text{whenever} \qquad Z_C(p(x))=0.
\]
\end{definition}
In the rest of the paper the term `multiplier sequence' will refer exclusively to a classical multiplier sequence. Also, by `applying a sequence to a function $f$', we simply mean the application of the operator $T=\seq{\gamma_k}$ to $f$; that is, if $f(x)=\sum_{k=0}^{\infty}a_k x^k$, then $T[f(x)]:=\sum_{k=0}^{\infty}\gamma_k a_k x^k$.
 The following is one of the essential results concerning the characterization of multiplier sequences, due to P\'olya and Schur.
\begin{theorem}(\cite{PS} or \cite[Ch.\,II.]{obreschkoff})\label{algchar} Let $\seq{\gamma_k}$ be a sequence of real numbers. The following are equivalent:
\begin{itemize}
\item[(i)] $\seq{\gamma_k}$ is a multiplier sequence;
\item[(ii)] (Algebraic characterization) for each $n \in \mathbb{N}_0$;
\[
\sum_{k=0}^n \binom{n}{k}\gamma_kx^k \in \LP I;
\]
\item[(iii)] (Transcendental characterization) 
\[
\sum_{k=0}^{\infty} \frac{\gamma_k}{k!}x^k \in \LP I.
\]
\end{itemize}
\end{theorem} 
\begin{definition}\label{jensenpoly} Let $T=\seq{\gamma_k}$ be a sequence of real numbers. For $n \in \mathbb{N}_0$, we definite the $n^{th}$ Jensen polynomial associated with the sequence $T$ to be
\[
g_n(x):=T[(1+x)^n]=\sum_{k=0}^n \binom{n}{k}\gamma_kx^k.
\]
\end{definition}
Given Theorem \ref{algchar}, a reasonably easy way to show that a sequence $T=\seq{\gamma_k}$ is {\it not} a multiplier sequence is to demonstrate the existence of a Jensen polynomial associated with $T$ possessing non-real zeros. \\
 \indent The following problem motivated most of the investigations in the present paper.
\begin{problem}\label{besseltypeproblem} Characterize all non-negative sequences $\seq{\gamma_k}$ such that if 
\[
f(x)=\sum_{k=0}^{\infty} \frac{\gamma_k}{k!}x^k
\]
is an entire function, then 
\begin{equation}\label{BesselGenGeorge}
F_p(x)=\sum_{k=0}^{\infty} \frac{\gamma_k}{k!\Gamma(k+p+1)}x^k \in \LP, \qquad \text{for} \quad -p \notin \mathbb{N}.
\end{equation}
\end{problem}
We regard the function $F_p(x)$ in (\ref{BesselGenGeorge}) as a generalized Bessel-type function. In support of this view, we recall that the {\it modified Bessel function of the first kind of order $p$} (\cite[p.\,228]{andrews} or \cite[p.\,116]{rs}) is defined as
\begin{equation}\label{I_p}
I_p(x):=\left( \frac{x}{2}\right)^p \sum_{k=0}^{\infty} \frac{(x/2)^{2k}}{k! \Gamma(k+p+1)}=\frac{(x/2)^p}{\Gamma(1+p)} {}_0F_1\left(-;1+p;\frac{x^2}{4}\right), \quad (-p \notin \mathbb{N}),
\end{equation}
where $\displaystyle{{}_0F_1\left(-;b;x\right):=\sum_{k=0}^{\infty} \frac{x^k}{(b)_k k!}}$ is the hypergeometric function, and $\displaystyle{(b)_p:=b(b+1) \cdots (b+n-1)=\frac{\Gamma(b+n)}{\Gamma(b)}}$ is the rising factorial, $-b \notin \mathbb{N}$, $n \in \mathbb{N}$ and $(b)_0=1$. Simple transformations show that with $\gamma_k=1/2^{2k}$ for $k=0,1,2,\ldots$, the function $F_0(x)$ in (\ref{BesselGenGeorge}) reduces to the modified Bessel function $I_0(\sqrt{x})$. We emphasize here that our generalizations of the Bessel functions are different from those appearing in the literature. Indeed, see for example \'A. Baricz' excellent monograph \cite{baricz}, where he studies, for suitable parameters $b$ and $c$, the function 
\[
w_p(z)=\sum_{n=0}^{\infty} \frac{(-c)^n}{n!\Gamma(p+n+(b+1)/2)}\left( \frac{z}{2}\right)^{2n+p},
\]
and refers to it as the generalized Bessel function of the first kind of order $p$. \\
\indent In reference to Problem \ref{besseltypeproblem}, it is clear that $F_p(x) \in \LP$ whenever $\seq{\gamma_k}$ is a multiplier sequence. Thus, the task is to characterize non-negative real sequences $\seq{\alpha_k}$, which are not multiplier sequences, but for which the `composed' sequence $\seq{\alpha_k/k!}$ is a multiplier sequence. Canonical examples appear to be difficult to construct. As an illustrative example, the sequence $\seq{k^2+2}$ is not a multiplier sequence, since $\displaystyle{\sum_{k=0}^{\infty} \frac{k^2+2}{k!}x^k=e^x(2+x+x^2) \notin \LP}$ (cf. (iii), Theorem \ref{algchar}). On the other hand, one can readily check that 
\[
F(x)=\sum_{k=0}^{\infty}\frac{k^2+2}{k!k!}x^k=(2+x)I_0(\sqrt{x})={}_0F_1\left(-;1;x\right) \in \LP,
\]
and whence $\seq{\frac{k^2+2}{k!}}$ is a multiplier sequence  (see Proposition \ref{quad_by_fact}).\\
\indent The rest of the paper is organized as follows. In connection with Problem 8, Section 2 investigates a logarithmically interpolated sequence (Theorem 11 and Corollary 13),  and  sums and averages of multiplier sequences (Theorem 14 and Corollary 16). By adopting some of the ideas of Hardy \cite{hardy_paper} and Ostrovskii \cite{ostrovskii}, the main results of Section 3 furnish several integral representations of the entire function
$f(x)=\sum_{k=0}^{\infty} \frac{\sqrt{k}}{k!k!}x^k$ (cf. Theorem 22). Motivated by the work in Section 3 (see, in particular, Example 20), Section 4 provides  generating functions which yield  new families of multiplier sequences varying smoothly with a parameter. The goal of Section 5 is multifold: (i) to indicate possible applications of the foregoing results in the theory Bessel functions or hypergeometric functions, (ii) to highlight additional propositions  (see, for example,  Proposition \ref{quad_by_fact} ) supporting the conjecture in Section 3.1 and (iii) to cite additional examples and list problems which arose during the analysis  of various sequences, but remain unsolved at this time.

\section{The log sequence}\label{log}
The function $B_{s,a}(z)$ in equation (\ref{genBessel}) can be regarded as a `generalized' exponential function \'a la Hardy, whose Taylor coefficients are interpolated by the function $\displaystyle{g(x)=\frac{(x+a)^s}{\Gamma(x+1)}}$, $s \in \mathbb{R}, a \geq 0$. The restriction on $a$ assures that for $s \in \mathbb{N}$, $g(x)$ interpolates a multiplier sequence. Notice that for non-integral $s >0$, $g(x)$ is not entire. Thus, our  explorations differ from the traditional approach, where the interpolating function is almost exclusively taken to be entire. We first look at a logarithmically interpolated sequence and the following real entire function:
\[
f(x)=\sum_{k=0}^{\infty} \frac{\ln(k+2)}{k!k!}x^k.
\]
When understood as an alteration of the modified Bessel function of the first kind of order zero (see (\ref{I_p}))
\[
I_0(x)=\sum_{k=0}^{\infty} \frac{(x/2)^{2k}}{k!k!},
\]
 one would attempt to establish the reality of zeros of $f$ by showing that $\seq{\ln(k+2)}$ is a multiplier sequence. This is not the case, however. If $T:=\{\ln (k+2)\}_{k=0}^{\infty}$, then the zeros of the Jensen polynomial
$$
g_3(x)=T[(1+x)^3]=\ln 2+3x\ln 3 +3x^2\ln 4 +x^3\ln 5
$$
are $x_3= - 0.330544\dots$ and $x_{1,2}=-1.1267576\dots\pm i\, 0.182619129\dots$.
We are thus led to consider the new sequence $T:=\{\ln\, (k+2)/k!\}_{k=0}^{\infty}$ and the associated entire function:
$$
f(x):=T[e^x]=T\left[\sum_{k=0}^{\infty} \frac{x^k}{k!}\right]=\sum_{k=0}^{\infty} \left(\frac{\ln (k+2)}{k!} \right)\frac{x^k}{k!}.
$$

While we believe that the sequence $\seq{\ln(k+2)/k!}$ is a multiplier sequence, we were able to establish such a claim only for an approximating sequence. In order to be able to formulate our theorem (cf. Theorem \ref{approxthm}), we need a few preliminary results. Recall (see \cite[p. 8]{rainville}) that 
\[
\lim_{n \to \infty}(H_n-\ln(n))=\gamma,
\]
where $\displaystyle{H_n:=\sum_{k=1}^n \frac{1}{k}}$ is the $n^{th}$ harmonic number, and $\gamma$ is the Euler-Mascheroni constant. Thus, for $n \gg 1$, 
\[
H_{n+2}-\gamma \approx \ln(n+2),
\]
and consequently,
\[
\frac{H_{n+2}-\gamma}{n!} \approx \frac{\ln(n+2)}{n!}.
\]
\begin{proposition} If $n \in \mathbb{N}$, then
\[
 H_n=\sum_{k=1}^n \binom{n}{k} (-1)^{k-1} \frac{1}{k}.
\]
\end{proposition}
\begin{proof} The proof is based on an induction argument, treating the even and odd cases separately.
\end{proof}

The following result is well known, but for the sake of completeness, we include a short proof of it here.
\begin{proposition} (\cite[vol.\,I, p.\,15]{emot}\label{digamma}). Let $\psi(x)$ denote the digamma function
\[
\psi(x):=\frac{d}{dx} \ln \Gamma(x)=\frac{\Gamma'(x)}{\Gamma(x)}.
\]
Then
\[
\sum_{n=0}^{\infty}\frac{H_{n+2}}{n!}\frac{x^n}{n!}=\sum_{n=0}^{\infty} \frac{\gamma+\psi(3+n)}{n!}\frac{x^n}{n!}.
\]
\end{proposition}
\begin{proof} It is well known (see \cite[pp.\,11-12]{rainville}) that 
\begin{eqnarray*}
\Gamma'(1)&=&-\gamma, \qquad \text{and}\\
\Gamma(x+1)&=&x \Gamma(x), \quad \text{for all} \quad x >0.
\end{eqnarray*}
It follows that
\[
\psi(x+1)=\frac{d}{dx} \ln \Gamma(x+1)=\frac{1}{x}+\frac{d}{dx} \ln \Gamma(x)=\frac{1}{x}+\psi(x).
\]
Starting with $\psi(2)=1+\psi(1)=H_1-\gamma$, a simple inductive argument establishes that $\displaystyle{\psi(n+1)=H_n-\gamma}$ for all $n \in \mathbb{N}$, which completes the proof.
\end{proof}
\begin{theorem}\label{approxthm} The entire function 
\[
f(x):=\sum_{n=0}^{\infty} \frac{(H_{n+2}-\gamma)}{n!}\frac{x^n}{n!}
\]
belongs to $\LP^+$, and hence the sequence $\left\{\frac{(H_{n+2}-\gamma)}{n!}\right\}_{n=0}^{\infty}$ is a multiplier sequence.
\end{theorem}
\begin{remark} Before proving Theorem \ref{approxthm}, we observe that the class of functions $\LP(-\infty,0]$ (see Definition \ref{LPI}) is not closed under differentiation. For example, $\varphi(x):=e^{-x^2+x}(x+1) \in \LP(-\infty,0]$, but $\varphi'(x)=e^{-x^2+x}(2-x-2x^2)$ has a positive zero.  
\end{remark}
\begin{proof} We note that $H_n-\gamma >0$ for all $n \in \mathbb{N}$ (see \cite[p.\,9]{rainville}), and hence the Taylor coefficients of $f$ are all positive. With the aid of Proposition \ref{digamma}, we can express $f(x)$ as
\begin{equation}\label{f1}
 f(x)=\sum_{n=0}^{\infty} \frac{\psi(n+3)}{n!} \frac{x^n}{n!}.
\end{equation} 
Since $\displaystyle{\frac{1}{\Gamma(x)} \in \LP(-\infty,0]}$ (cf. Definition \ref{LPI}), it follows that
\[
\left(\frac{1}{\Gamma(x)} \right)'=-\frac{\Gamma'(x)}{\Gamma^2(x)}=-\frac{\psi(x)}{\Gamma(x)} \in \LP.
\]  
It is known that for $x>0$, the only extremum of $\Gamma(x)$ occurs at $x_0=1.4616\ldots$ (see, for example, \cite[p.\,90]{andrews}). Since $x_0$ corresponds to a minimum of $\Gamma(x)$, we infer that $\Gamma'(x)$ and $\psi(x)$ are both negative on the interval $(0, x_0)$ and are both positive on $(x_0,\infty)$. It now follows that all the zeros of the entire function $\varphi(x):=\frac{\psi(x)}{\Gamma(x)}$ lie in $(-\infty, 2)$. Hence, $\varphi(x+3) \in \LP(-\infty,0)$ and consequently, by Laguerre's theorem (\cite[Theorem 4.1(3)]{ccczds}), the sequence $T:=\seq{\varphi(k+3)}$ is a CZDS (cf. Definition \ref{CZDS}) and {\it a fortiori} T is a multiplier sequence.  Thus
\[
T(e^x)=\sum_{k=0}^{\infty} \frac{\psi(k+3)}{\Gamma(k+3)}\frac{x^k}{k!}=\sum_{k=0}^{\infty} \frac{\psi(k+3)}{(k+2)!}\frac{x^k}{k!} \in \LP^+.
\]
Finally, applying the multiplier sequence $\seq{(k+2)(k+1)}$ to $T(e^x)$ yields the desired result (cf. (\ref{f1})):
\[
\sum_{k=0}^{\infty} \frac{(k+2)(k+1)\psi(k+3)}{(k+2)!}\frac{x^k}{k!}=\sum_{k=0}^{\infty} \frac{\psi(k+3)}{k!}\frac{x^k}{k!}=f(x) \in \LP^+.
\]
\end{proof}
\begin{corollary} For $t \in \mathbb{R}$, let $\{t \}=t-\lfloor t \rfloor$, where $\lfloor t \rfloor$ denotes the greatest integer less than or equal to $t$. Then the sequence
\[
\seq{\frac{\displaystyle{\ln (k+2)+ \int_{k+2}^{\infty} \frac{\{t\}}{t^2}dt}}{k!}}
\]
is a multiplier sequence.
\end{corollary}
\begin{proof} In \cite[p.\,540]{lagarias} J.~Lagarias states that for all $k \geq 1$, 
\begin{equation} \label{lagar}
H_k=\ln k + \gamma +\int_k^{\infty} \frac{\{t \}}{t^2}dt.
\end{equation}
Rearranging equation (\ref{lagar}) and applying Theorem \ref{approxthm} yields the result.
\end{proof}
\subsection{Sums and averages} The harmonic approximation to the logarithm motivates the study of sums and averages of initial segments of multiplier sequences (and sequences in general), and whether or not such derived sequences are again multiplier sequences. We begin by noting that if $\seq{\gamma_k}$ is a multiplier sequence, then the sequences
\begin{equation}\label{twoTs}
T_1 = \seq{\sum_{j=0}^k \gamma_j } \qquad \text{and} \qquad T_2 = \seq{\frac{1}{k+1}\sum_{j=0}^k \gamma_k }
\end{equation}
need not be multiplier sequences. Indeed, if $\seq{\gamma_k}=\seq{1/k!}$, then 
\begin{eqnarray*}
T_1[(1+x)^4]&=&1+8x+15 x^2+\frac{32}{3}x^3+\frac{64}{24}x^4 \notin \LP,\\
T_2[(1+x)^3]&=&1+3x+\frac{5}{2}x^2+\frac{2}{3}x^3 \notin \LP.
\end{eqnarray*}
The converse implication however is true, if the sequence $\seq{\gamma_k}$ can be interpolated by a polynomial with non-negative coefficients.
\begin{theorem}\label{coolstuff} For $k \in \mathbb{N}_0$ let $\gamma_k=p(k)$, where $\ds{p(x):=\sum_{j=0}^m a_jx^j}$, and $a_j \geq 0$. Set
\begin{eqnarray*}
S(k)&=&\sum_{j=0}^k \gamma_j, \qquad \text{and}\\
A(k)&=&\frac{\gamma_0+\gamma_1+\cdots+\gamma_k}{k+1}, \qquad k \geq 0.\\
\end{eqnarray*}
If the average sequence $\seq{A(k)}$ is a multiplier sequence, then so is the sequence $\seq{\gamma_k}$. 
\end{theorem}
\begin{proof} We shall arrive at the desired result by demonstrating that the function
\[
f(x)=\sum_{k=0}^{\infty} p(k) \frac{x^k}{k!}
\]
belongs to $\LP^+$. To this end consider
\begin{eqnarray*}
Q(x)&=&e^{-x} \sum_{n=0}^{\infty} A(n) \frac{x^n}{n!}\\
&=&e^{-x} \sum_{n=0}^{\infty} \left(\sum_{k=0}^n p(k) \right) \frac{1}{n+1} \frac{x^n}{n!} \in \LP, 
\end{eqnarray*}
where the membership in $\LP$ follows, since by assumption, $\seq{A(k)}$ is a multiplier sequence. Consequently, $xQ(x)$ and its derivative both belong to $\LP$. We now calculate
\begin{eqnarray*}
D\left[xQ(x) \right]&=&D\left[e^{-x} \sum_{n=0}^{\infty} \left(\sum_{k=0}^n p(k) \right)\frac{1}{n+1}\frac{x^{n+1}}{n!} \right] \hspace{2 in} \left(D:=\frac{d}{dx} \right)\\
&=&e^{-x} \left[\sum_{n=0}^{\infty}\left(\sum_{k=0}^n p(k) \right)\frac{x^n}{n!}-\sum_{n=0}^{\infty}\left(\sum_{k=0}^n p(k) \right)\frac{x^{n+1}}{(n+1)!} \right]\\
&=&e^{-x} \left[p(0)+\sum_{n=0}^{\infty}\left(\sum_{k=0}^{n+1} p(k) \right)\frac{x^{n+1}}{(n+1)!}-\sum_{n=0}^{\infty}\left(\sum_{k=0}^n p(k) \right)\frac{x^{n+1}}{(n+1)!} \right]\\
&=&e^{-x}\left[p(0) +\sum_{n=0}^{\infty} p(n+1) \frac{x^{n+1}}{(n+1)!}\right]\\
&=&e^{-x}\left[\sum_{n=0}^{\infty} p(n) \frac{x^{n}}{n!}\right]\\
&=&e^{-x} f(x).
\end{eqnarray*}
Thus $e^x D[xQ(x)]=f(x) \in \LP$. The assumption that $a_j \geq 0$ for all $j\in \mathbb{N}_0$ ensures that in fact $f \in \LP^+$, and our proof is complete.
\end{proof}
We offer two corollaries of Theorem \ref{coolstuff}.
\begin{corollary} If $\seq{A(k)}$ is a multiplier sequence, then so is $\seq{S(k)}$. 
\end{corollary}
\begin{proof} The result is immediate, since $\seq{(k+1)}$ is a multiplier sequence. 
\end{proof}
\begin{corollary} \label{shiftright} Suppose that $p$ is as in the statement of Theorem \ref{coolstuff}, and let $m=\deg p$. If $1 \leq \ell \leq m+1$ and the sequence $\displaystyle{\seq{\frac{S(k)}{(k+1)_{\ell}}}}$ is a multiplier sequence, then so is the sequence
\[
\{0,0,0,\ldots,\underbrace{p(0)}_{\ell\text{th slot}},p(1),p(2), \ldots \}.
\]
\end{corollary}
\begin{proof} The proof is essentially the same as that of Theorem \ref{coolstuff}, if one differentiates $x^{\ell}\widetilde{Q}(x)$, where $\widetilde{Q}(x)$ is an appropriately modified version of $Q(x)$. 
In particular, 
\begin{eqnarray*}
&&D\left[e^{-x}\left(x^{\ell} \sum_{k=0}^{\infty} \frac{S(k)}{(k+1)_{\ell}}\frac{x^k}{k!} \right) \right]\\
&=&D\left[e^{-x}\left( \sum_{k=0}^{\infty} S(k)\frac{x^{k+\ell}}{(k+\ell)!} \right) \right]\\
&=&e^{-x}\left[\sum_{k=\ell-1}^{\infty}p(k-\ell+1)\frac{x^k}{k!} \right] \in \LP,
\end{eqnarray*}
and hence the sequence $\{0,0,0,\ldots,p(0),p(1),p(2),\ldots\}$ is a multiplier sequence. 
\end{proof}
We remark that Corollary \ref{shiftright} gives a sufficient condition when one can pre-concatenate a polynomially interpolated multiplier sequence with a string of zeros and thus obtain another multiplier sequence. 

\begin{example} This example is an illustration of Corollary \ref{shiftright} under the assumption that
\[
p(x)=\prod_{j=1}^m(x+j) \in \LP^+, \qquad m \geq 1.
\]
For such polynomials we claim, that if $S(n):=\sum_{k=0}^m p(k)$, then 
\begin{equation}\label{special}
S(n)=\frac{1}{m+1}\prod_{k=1}^{m+1} (k+n) \qquad \text{for all} \ n \in \mathbb{N}.
\end{equation}
\begin{proof}[Proof of Claim] We proceed by double induction. Fix the degree $m$ of $p$, and let $n=1$.  
\[
S(1)=\sum_{k=0}^1 p(k)=m! +\prod_{k=1}^m(k+1)=m!(m+2)=\frac{1}{m+1} \prod_{k=1}^{m+1}(k+1).
\]
Suppose now that equation (\ref{special}) holds for some $n \geq 1$. Then 
\begin{eqnarray*}
S(n+1)&=&p(n+1)+S(n)\\
&=&\prod_{k=1}^m (n+1+k)+\frac{1}{m+1}\prod_{k=1}^{m+1} (k+n)\\
&=&\frac{1}{m+n+2}\prod_{k=1}^{m+1} (k+n+1)+\frac{n+1}{m+1}\frac{1}{m+n+2}\prod_{k=1}^{m+1} (k+n+1)\\
&=&\frac{1}{m+1} \prod_{k=1}^{m+1} (k+n+1).
\end{eqnarray*}
Since $m$ was arbitrary, the claim follows.
\end{proof}
From equation (\ref{special}) one can readily deduce that the sequence $\displaystyle{\seq{\frac{S(k)}{(k+1)_{\ell}}}}$ is a multiplier sequence for any $1 \leq \ell \leq m+1$. Corollary \ref{shiftright} implies that the sequences
\begin{eqnarray*}
&&\{p_{m}(0), p_{m}(1),p_m(2), \ldots \}\\
&&\{0,p_{m}(0),p_{m}(1),p_{m}(2),\ldots \}\\
&&\{0,0,p_{m}(0), p_{m}(1),p_{m}(2),\ldots\}\\
&& \vdots\\
&& \{0,0,\ldots,\underbrace{p_m(0)}_{m+1\text{st slot}},p_m(1),\ldots\}
\end{eqnarray*}
are all multiplier sequences. By way of illustration, if $m=4$ and $\ell=2$, one obtains the multiplier sequence
\[
\{0,0,4!,5!,\frac{6!}{2!},\frac{7!}{3!},\ldots,\frac{(k+2)!}{(k-2)!},\ldots \},
\]
and thus, the entire function
\[
\sum_{k=2}^{\infty} \frac{(k+2)!}{(k-2)!}\frac{x^k}{k!}=e^xx^2(x+2)(x+6) \in \LP^+.
\]
\end{example}
\begin{example} \label{conversefalse} The converse of Theorem \ref{coolstuff} is false in general. That is, if $\seq{p(k)}$ is a multiplier sequence, the average $\seq{A(k)}$ need not be a multiplier sequence. The sequence $\seq{p(k)}=\seq{1+k+k^2}$ is a multiplier sequence, since
\[
\sum_{k=0}^{\infty} \frac{1+k+k^2}{k!}x^k=e^x(1+x)^2 \in \LP^+.
\]
 The average sequence $\displaystyle{\seq{A(k)}=\seq{\frac{1}{3}(3+2 k+k^2)}}$ however is not a multiplier sequence, because
\[
\sum_{n=0}^{\infty} \frac{3+2 n+n^2}{3} \frac{x^n}{n!}=\frac{1}{3}e^x(3+3x+x^2) \notin \LP.
\]
As can be verified by the reader, remarkably, both the sequence $\seq{(1+k+k^2)^3}$ and its average 
\[
\displaystyle{\seq{\frac{1}{105}(105+244k+386 k^2+384 k^3+246 k^4+90 k^5+15 k^6)}}
\]
are multiplier sequences. \\
\indent One may wonder whether requiring $p$ to belong to $\LP^+$ could result in a partial converse of Theorem \ref{coolstuff}. This is not the case. Setting $p(x)=(x+4)^2$, we see that $S(k)=\frac{1}{6}(1+k)(96+25k+2 k^2)$, and 
\[
\sum_{k=0}^{\infty} \frac{S(k)}{k!}x^k=\frac{e^x}{6}(96 + 150 x + 33 x^2 + 2 x^3) \notin \LP.
\]
\end{example}

\section{$\seq{\frac{(k+a)^s}{k!}}$ type sequences}\label{generalizedbessel}
In general, the sequence $\seq{\displaystyle{\frac{1}{(k+a)k!}}}$ is not a multiplier sequence. For example, if $a=1/2$, then the Jensen polynomial (cf. Definition \ref{jensenpoly})
\[
g_4(x)=\sum_{k=0}^4 \binom{4}{k} \frac{x^k}{(k+1/2)k!}=2+\frac{8}{3}x+\frac{6}{5}x^2+\frac{4}{21}x^3+\frac{1}{108}x^4
\]
has two non-real zeros. 
\begin{lemma}(\cite[Proposition 40]{chasse}) \label{ostrov1} The sequence $\seq{\displaystyle{\frac{1}{(k+a)k!}}}$ is a multiplier sequence for every $a \in \mathbb{N}$.
\end{lemma}
\begin{proof} The result follows immediately from the fact that for any $a \in \mathbb{N}$, 
\begin{eqnarray*}
\frac{(k+1)(k+2)\cdots(k+a-1)}{\Gamma(k+a+1)}&=&\frac{(k+1)(k+2)\cdots(k+a-1)}{1\cdot2\cdots k (k+1)(k+2) \cdots (k+a)}\\
&=&\frac{1}{k!(k+a)}.
\end{eqnarray*}
\end{proof}
\subsection{The sequence $\seq{\sqrt{k}/k!}$}\label{rootkoverkfact} We conjecture that the function $\displaystyle{\varphi(x)=\sum_{k=0}^{\infty} \frac{\sqrt{k}}{k!}\frac{x^k}{k!}}$ belongs to $\LP^+$. Our contention is that among the sequences of the form $\seq{k^s/k!}$ with $s$ non-integral, the case $s=1/2$ is special. In meager support of this claim, we consider the following examples.
\begin{example} The sequence $\seq{k^{1/20}/k!}$ is not a multiplier sequence. In fact, if we consider the Jensen polynomials associated with
\[
f(x)=\sum_{k=0}^{\infty} \frac{k^{1/20}}{k!}x^k
\]
we find that the sixth Jensen polynomial
\[
g_6(x)=6x+\frac{15 x^2}{2^{19/20}}+\frac{10 x^3}{3^{19/20}}+\frac{5x^4}{40 \cdot 2^{9/10}}+\frac{x^5}{4 5^{19/20}}+\frac{x^6}{120 \cdot 6^{19/20}}
\]
has has only four real zeros, along with a pair of non-real zeros. This phenomenon persists for small $s$, but appears to change when $s=1/2$. In this case all the Jensen polynomials that we tested have only real zeros.
\end{example}

There is a marked sparsity of known multiplier sequences which involve $\sqrt{k}$ non-trivially. We offer here the following examples. 
\begin{itemize}
\item[(i)] The sequence $\seq{\cosh(\sqrt{k})}$ is a multiplier sequence. This follows from the containment
\[
\cosh{\sqrt{x}}=\sum_{k=0}^{\infty} \frac{x^k}{(2k)!}=\sum_{k=0}^{\infty} \frac{k!}{(2k)!}\frac{x^k}{k!}=\prod_{k=0}^{\infty} \left(1+\frac{x}{\left(\pi k+\frac{\pi}{2} \right)^2} \right) \in \LP^+,
\]
together with Laguerre's theorem (\cite[Theorem 4.1(3)]{ccczds}). 
\item[(ii)] In order to formulate the second example, we first recall that a sequence $\seq{\gamma_k}$ of non-negative real numbers is said to be \emph{rapidly decreasing}, if $\gamma_k^2 \geq 4 \gamma_{k-1} \gamma_{k+1}$ for all $k \in \mathbb{N}$ (\cite[p.438]{ccczds}). Such sequences are known to be multiplier sequences. We now note that if $\seq{\gamma_k}$ is rapidly decreasing, then so is $\seq{\sqrt{k} \gamma_k}$, which in turn makes the latter also a multiplier sequence.
\end{itemize}
It is not known whether $\displaystyle{\varphi(x)=\sum_{k=0}^{\infty} \frac{\sqrt{k}}{k!}\frac{x^k}{k!}}$ belongs to $\LP^+$. The entire function $\varphi(x)$ is however Hurwitz stable; that is, all of its zeros lie in the closed left half-plane. 
\begin{lemma} The entire function $\displaystyle{\varphi(x)=\sum_{k=0}^{\infty} \frac{\sqrt{k}}{k!}\frac{x^k}{k!}}$ is Hurwitz stable.
\end{lemma}
\begin{proof} The work of Ostrovskii (\cite[Corollary 2.2]{ostrovskii}) shows that the entire function
\[
f(x)=\sum_{k=0}^{\infty} \frac{\sqrt{k}}{k!}x^k
\]
is Hurwitz stable. Thus, $f(-ix)$ has all of its zeros in the closed upper half-plane, and hence by the Hermite-Biehler theorem, $f(-ix)=p(x)+iq(x)$, where $p(x)$ and $q(x)$ have only real, interlacing zeros. Let $T=\seq{1/k!}$. Then the entire function $T[f(-ix)]=T[p(x)]+iT[q(x)]$ also has all its zeros in the upper half-plane \cite[p.\,342]{levin}. Finally, the change of variables $x \mapsto ix$ shows that $\varphi(x)$ is Hurwitz stable.
\end{proof}
We close this section by giving two integral representations for the function $\displaystyle{\varphi(x)=\sum_{k=0}^{\infty} \frac{\sqrt{k}}{k!}\frac{x^k}{k!}}$, which may help in determining whether it belongs to $\LP^+$. We arrive at the first of the two representations by adapting the main ideas in \cite[Section 3]{ostrovskii}. 
\begin{theorem}\label{intrep1} If $\displaystyle{f(x,t):=\sum_{n=0}^{\infty} \frac{x^n (e^{-t})^n}{n!n!}}$, then 
\begin{equation}\label{squareroot}
\sum_{n=0}^{\infty} \frac{ \sqrt{n}}{n!}\frac{x^n}{n!}=-\frac{1}{ 2 \sqrt{\pi}} \int_0^{\infty} \left[f(x,u)-f(x,0) \right]\frac{du}{u^{3/2}}.
\end{equation}
\end{theorem}
\begin{proof}
We start with the following generalizations of the modified Bessel function:
\begin{eqnarray} \label{genbess}
B(0,x)&:=&\sum_{n=0}^{\infty} \frac{x^n}{n!n!},  \nonumber \\
B(s,x)&:=&\sum_{n=0}^{\infty} \frac{n^s}{n!}\frac{x^n}{n!} \qquad (s >0).
\end{eqnarray}
Differentiating (\ref{genbess}) with respect to $x$ yields
\[
\frac{d}{dx}B(s,x)=\sum_{n=1}^{\infty} \frac{ n^{s+1}}{n!}\frac{x^{n-1}}{n!},
\]
and whence
\[
x\frac{d}{dx}[B(s,x)]=B(s+1,x), \qquad s \geq 0.
\]
 Now set $\displaystyle{f(x,t):=\sum_{n=0}^{\infty} \frac{x^n (e^{-t})^n}{n!n!}}$, and consider the generating relation
\begin{equation} \label{bessgenerating}
f(x,t)=\sum_{n=0}^{\infty}(-1)^n Q_n(x) \frac{t^n}{n!}.
\end{equation}
Differentiating (\ref{bessgenerating}) with respect to $t$ and $x$ yields
\begin{eqnarray}
\sum_{n=0}^{\infty} \frac{-n x^n (e^{-t})^n}{n! n!}&=&\sum_{n=0}^{\infty}(-1)^{n+1} Q_{n+1}(x) \frac{t^n}{n!} \label{tdiv}, \quad \text{and}\\
\sum_{n=0}^{\infty} \frac{ n x^{n-1}(e^{-t})^n}{n!n!}&=&\sum_{n=0}^{\infty} (-1)^n Q'_n(x) \frac{t^n}{n!}. \label{xdiv}
\end{eqnarray}
By equating the coefficients of $t^n$ in (\ref{tdiv}) and (\ref{xdiv}), we deduce that 
\[
Q_{n+1}(x)=xQ'_n(x), \qquad n \in \mathbb{N}_0.
\]
Notice that $Q_0(x)=B(0,x)$, and since the sequences $\{ B(n,x)\}$ and $\{Q_n(x) \}$ satisfy the same recurrence relation, we conclude that 
\[
Q_n(x)=B(n,x), \qquad n \in \mathbb{N}_0.
\]
With the aid of (\ref{bessgenerating}) we now give an integral representation for $B(s,x)$ for non-integral values of $s$. For $k<s<k+1$, $k \in \mathbb{N}_0$, the Cauchy-Saalsch\"utz formula (\cite[Sec. 12.21]{whittaker}) yields
\[
\Gamma(-s)=\int_0^{\infty} \left[e^{-t}-\sum_{j=0}^k (-1)^j \frac{t^j}{j!} \right]\frac{dt}{t^{s+1}}.
\] 
The change of variables $t=nu$ gives
\[
n^s=\frac{1}{\Gamma(-s)}\int_0^{\infty} \left[e^{-nu}-\sum_{j=0}^k (-1)^j \frac{(nu)^j}{j!} \right]\frac{du}{u^{s+1}}.
\]
Consequently,
\begin{eqnarray} \label{integralrep}
B(s,x)&=&\sum_{n=1}^{\infty} \frac{n^s}{n!}\frac{x^n}{n!}\\
&=&\sum_{n=1}^{\infty} \frac{x^n}{n!n!} \frac{1}{\Gamma(-s)}\int_0^{\infty} \left[e^{-nu}-\sum_{j=0}^k (-1)^j \frac{(nu)^j}{j!} \right]\frac{du}{u^{s+1}} \nonumber \\
&=&\frac{1}{\Gamma(-s)}\int_0^{\infty} \left[f(x,u)-\sum_{j=0}^k (-1)^j \frac{Q_j(x)}{j!}u^j \right]\frac{du}{u^{s+1}}, \qquad (k<s<k+1, \quad k \in \mathbb{N}_0). \nonumber
\end{eqnarray}
Thus, setting $k=0$ and $s=1/2$ in (\ref{integralrep}) finishes the proof.
\end{proof}
\begin{corollary} The following representation is valid:
\begin{equation} \label{intrep}
\sum_{n=0}^{\infty} \frac{\sqrt{n}}{n!}\frac{x^n}{n!}=\frac{1}{2 \sqrt{\pi}} \int_0^1 \left[B(0,x)-B(0,xv) \right] \frac{dv}{v (-\ln(v))^{3/2}}
\end{equation}
\end{corollary}
\begin{proof} The above representation follows directly from (\ref{squareroot}) with the change of variables $v=e^{-u}$. We remark that the convergence of the integral for every $x \in \mathbb{C}$ in the statement of the corollary can be directly verified using the identity
\[
\int_0^1 \frac{1-v^n}{v(-\ln(v))^{1+s}} dv=-n^s \Gamma(-s), \qquad n \in \mathbb{N}, \ 0<s<1.
\] 
\end{proof}
We conclude this section with two more integral representations which could be of use in further investigations. With $I_p$ denoting the modified Bessel functions (cf. (\ref{I_p})), the following formul\ae
\[
\varphi(x):=\sum_{n=0}^{\infty} \frac{\sqrt{n}}{n!}\frac{x^n}{n!}=\frac{1}{\sqrt{\pi}} \int_0^1 \frac{\sqrt{x} I_1(2 \sqrt{x t})}{\sqrt{t} \sqrt{-\ln t}}dt,
\]
and
\[
\varphi'(x)=\frac{1}{\sqrt{\pi}}\int_0^1 \frac{I_0(2 \sqrt{x t})}{\sqrt{- \ln t}}dt.
\]

 \section{Transformations of multiplier sequences} \label{transformations}This section is motivated by the observation that in some sense the sequence $\seq{\displaystyle{(\ln(k+2))/k!}}$ `lies between' two multiplier sequences (namely $\seq{1}$ and $\seq{\displaystyle{1/k!}}$), since $k! \succ \ln(k+2) \succ 1$ in terms of growth. By Theorem \ref{approxthm} the sequence $\seq{(H_{k+2}-\gamma)/k!}$ is a multiplier sequence, which itself is an approximation to the sequence $\seq{(\ln(k+2))/k!}$. Thus the notions of deformation, perturbation and general transformation of multiplier sequences arise naturally. A quite fruitful way of obtaining new multiplier sequences is to identify those transformations, which when applied to multiplier sequences, result again in multiplier sequences. We formulate the following general problem.
 
 \begin{problem}\label{birthofnewsequences} Let $\seq{\alpha_k}$ and $\seq{\beta_k}$ be multiplier sequences of non-negative real numbers. Characterize all functions $\Psi:\R \times \R \to \R$ such that $\seq{\Psi(\alpha_k,\beta_k)}$ is again a multiplier sequence. 
\end{problem}
There are some simple functions with this property, such as the projection onto either coordinate axis, or the function $\Psi(x,y)=xy$. 
\begin{lemma} Convex combinations of two multiplier sequences need not be multiplier sequences. That is,  in general $\Psi_{\lambda}(\alpha_k, \beta_k):=\lambda \alpha_k+ (1-\lambda) \beta_k$ ($\lambda \in [0,1]$) is not a solution to Problem \ref{birthofnewsequences}.
\end{lemma}
\begin{proof} Let$\seq{\alpha_k}=\seq{k^2+k+1}$ and $\seq{\beta_k}=\seq{\frac{1}{k!}}$. With $\lambda=\frac{1}{10}$, we get the sequence $\seq{\frac{k^2+k+1}{10}+\frac{9}{10 k!}}$, which, when applied to $(1+x)^4$ yields the polynomial
\[
p(x)=1+\frac{24}{5}x+\frac{69}{10}x^2+\frac{29}{5}x^3+\frac{171}{80}x^4 \notin \LP.
\]
\end{proof}
Simple examples show that there exists multiplier sequences of non-negative terms, whose linear combination is again a multiplier sequence.
\begin{proposition} (\cite[p.\,198]{rs}) Suppose that $p,q \in \LP^+$ are polynomials with strictly interlacing zeros. Then for any $a,b \in \R^+$, the polynomial $a q(x)+b p(x) \in \LP^+$. Thus, any positive linear combination of the multiplier sequences $\seq{q(k)}$ and $\seq{p(k)}$ is again a multiplier sequence. 
\end{proposition} 

\begin{lemma} Convex geometric combinations of two multiplier sequences need not be multiplier sequences. That is,  in general $\Psi_{\lambda}(\alpha_k, \beta_k):= \alpha_k^{\lambda}\beta_k^{(1-\lambda)}$, $\lambda \in [0,1]$, is not a solution to Problem \ref{birthofnewsequences}.
\end{lemma}
\begin{proof} Consider $\seq{\alpha_k}=\seq{k^2+k+1}$ and $\seq{\beta_k}=\seq{1}$. Calculating $\Psi_{1/2}(\alpha_k,\beta_k)[(1+x)^4]$ gives the polynomial
\[
q(x)=1+4 \sqrt{3}x+6\sqrt{7}x^2+4\sqrt{13}x^3+\sqrt{21}x^4,
\]
which has two non-real zeros.
\end{proof}
The following proposition provides a `continuously deformed' family of multiplier sequences.
\begin{proposition}\label{parametersequence} Suppose that 
\[
\varphi(x)=\sum_{k=0}^{\infty} \frac{\gamma_k}{k!}x^k \in \LP^+.
\]
Then for $t \in [0,1]$ the sequence $\seq{B^{\varphi}_k(t)}$ is a multiplier sequence, where
\begin{equation}\label{Bstuff}
B^{\varphi}_k(t)=\sum_{j=0}^k \binom{k}{j}(1-t)^j \gamma_{k-j} t^{k-j}.
\end{equation}
\end{proposition}
\begin{proof} This is a straightforward consequence of the generating relation
\[
e^{(1-t)x}\varphi(xt)=\sum_{k=0}^{\infty} \frac{B^{\varphi}_k(t)}{k!}x^k.
\]
\end{proof}
Note that $\seq{B^{\varphi}_k(0)}=\seq{\gamma_0}$ and $\seq{B^{\varphi}_k(1)}=\seq{\gamma_k}$.
\begin{corollary}\label{twoparametersequence} Let $\varphi, \Phi \in \LP^+$, and let $\seq{B^{\varphi}_k(t)}$ and $\seq{B^{\Phi}_k(s)}$ be the associated multiplier sequences (cf. equation (\ref{Bstuff})). Then the sequence
\begin{equation}\label{C_k}
\seq{C^{\varphi,\Phi}_k(t,s)}:=\seq{\sum_{j=0}^k \binom{k}{j} B^{\varphi}_j(t) B^{\Phi}_{k-j}(s)}
\end{equation}
is a multiplier sequence for all $(t,s) \in [0,1] \times [0,1]$.
\end{corollary}
\begin{proof} The representation (\ref{C_k}) is a consequence of the generating relation
\begin{equation} \label{2parametergenerating}
e^{((1-t)+(1-s))x} \varphi(xt)\Phi(xs)=\sum_{k=0}^{\infty} \frac{C^{\varphi,\Phi}_k(t,s)}{k!}x^k.
\end{equation}
\end{proof}
\begin{remark} Several observations are in order.
\begin{itemize}
\item[(a)] If $f:[0,1] \to [0,1]$, then the sequence $\{C_k^{\varphi,\Phi}(t,f(t))\}_{k=0}^{\infty}$ is a multiplier sequence for all $t \in [0,1]$. 
\item[(b)] Proposition \ref{parametersequence} and Corollary \ref{twoparametersequence} involve Cauchy products with parameters. Let $\ds{\varphi (t)=\sum_{k=0}^{\infty} \frac{\gamma_k}{k!}x^k}\in \LP^+$ and $\ds{\Phi (t)=\sum_{k=0}^{\infty} \frac{\beta_k}{k!}x^k}\in \LP^+$. Then the Cauchy product  of $\varphi(t)$ and $\Phi(t)$ is also in $\LP^+$; that is,
\[
\varphi(t)\, \Phi(t)= \sum_{k=0}^{\infty} c_k\frac{t^k}{k!} \in \LP^+,\qquad{\text where} \qquad c_k:=\sum_{j=0}^k \binom{k}{j} \gamma_j\,\beta_{k-j},
\]
and hence $\{c_k\}_{k=0}^{\infty}$ is a multiplier sequence. 
\item[(c)] The polynomial $B_k^{\varphi}(t)$ (see (4.1)) can also be expressed in terms of the Jensen polynomials $\{g_j(t)\}_{j=0}^{\infty}$ associated with $\varphi (t)$:
\[
B_k^{\varphi}(t)=\sum_{j=0}^k\binom{k}{j} g_j(t)(-1)^{j+k}t^{k-j}.
\]
\item[(d)] Finally, we remark that for each fixed $k\in \Bbb N$, the polynomial $B_k^{\varphi}(t)$ has only real zeros.  A short proof of this assertion is as follows. For fixed $k\in \Bbb N$ and $t\not= 0$, a calculation shows that
\[
t^k B_k^{\varphi}\left(\frac{1}{t}\right)=t^k\sum_{j=0}^k \binom{k}{j} \left(1-\frac{1}{t}\right)^j\gamma_{k-j}\frac{1}{t^{k-j}}=\sum_{j=0}^k  \binom{k}{j} (t-1)^j\gamma_{k-j},
\]
and whence we infer that  $B_k^{\varphi}(t)$ has only real zeros. Here a caveat is in order since, in general,  the zeros of  $B_k^{\varphi}(t)$ need not be all negative.
\end{itemize}
\end{remark}
\begin{definition}\label{c_krep} Let $\gamma_k \in \mathbb{R}$ for $k=0,1,2,\ldots$. We say that the sequence $\seq{\gamma_k}$ has a $C_k$-representation, if there exist functions $\varphi, \Phi \in \LP^+$ (not necessarily distinct) and $s,t \in \mathbb{R}$ such that $\gamma_k=C_k^{\varphi,\Phi}(t,s)$ for all $k \in \mathbb{N}_0$, where $C_k^{\varphi,\Phi}(t,s)$ is as defined in equation (\ref{C_k}).
\end{definition}
\begin{theorem} 
(1) Every polynomially interpolated multiplier sequence of non-negative terms  can be written as a sequence $\displaystyle{\seq{C_k^{\varphi,\Phi}(t,s)}}$ for some functions $\varphi, \Phi \in \LP^+$ and $(t,s) \in [0,1] \times [0,1]$. The choice of $\varphi, \Phi, t$ and $s$ in this representation need not be unique. \newline
(2) Every geometric multiplier sequence (i.e., a sequence of the form $\seq{r^k}$, $r \in \mathbb{R}$) has a $C_k$-representation.
\end{theorem}
\begin{proof} (1) For any $p \in \mathbb{R}[x]$,  $\displaystyle{\sum_{k=0}^{\infty} \frac{p(k)}{k!}x^k}=\widetilde{p(x)}e^x$, where
\[
\widetilde{p(x)}=a_0+\sum_{j=1}^n \left( \sum_{k=j}^n a_k S_2(k,j)\right) x^j \in \LP^+,
\]
and $S_2(k,j)$ denote Stirling numbers of the second kind (see, for example, \cite[Ch.7]{moll}).  Suppose now that $\seq{\gamma_k}$ is a multiplier sequence of non-negative terms, and suppose that $p \in \mathbb{R}[x]$ is such that $p(k)=\gamma_k$ for $k=0,1,2,\ldots$. The equation \begin{equation} \label{identification}
e^{((1-t)+(1-s))x} \varphi(xt)\Phi(xs)=\widetilde{p(x)}e^x
\end{equation}
leads to the identification $\seq{\gamma_k}=\seq{C_k^{\varphi,\Phi}(t,s)}$. Indeed, equation (\ref{identification}) will be satisfied if (a) $1=t+s$, and (b) $\varphi(xt)\Phi(xs)=\widetilde{p(x)}$.
The selection $t=1,s=0$, $\Phi(x) \equiv 1$ and $\varphi(x)=\widetilde{p(x)}$ obviously satisfies $(a)$ and $(b)$. Thus, $\seq{\gamma_k}=\seq{C_k^{\widetilde{p(x)},1}(1,0)}$. The non-uniqueness is easy to ascertain, for any $\widetilde{p(x)} \in \LP^+$ of degree two or higher has distinct factorizations into products of the form $\widetilde{p(x)}=q_1(x)q_2(x)$, with $q_1(x),q_2(x) \in \LP^+$.
\newline (2) Again, solving $(\ref{identification})$ with $e^{rx}$ on the right hand side is possible with $\varphi \equiv \Phi \equiv 1$ and $t+s=2-r$. In particular, $\seq{r^k}=\seq{C_k^{1,1}(t,2-t-r)}$. We remark that the restrictions on $t,s$ force $r \in [0,2]$. Geometric sequences with bases greater than two can still be produced, with the appropriate choice of $\varphi$ and $\Phi$.
\end{proof}
\begin{example} Suppose that $\varphi(x)=\Phi(x)$. In this case we have 
\[
C_k^{\varphi, \varphi}(t,s)=\sum_{j=0}^k \binom{k}{j} B_j^{\varphi}(t) B_{k-j}^{\varphi}(s), \quad s,t \in \R.
\]
For $t,s \in (0,1)$, set
\[
c_1=\frac{t}{1-t}, \qquad c_2=\frac{s}{1-s}, \qquad \text{and} \quad c_3=\frac{1-t}{1-s}.
\]
Then
\begin{equation}\label{examplephi}
C_k^{\varphi, \varphi}(t,s)=(1-s)^k \sum_{j=0}^k \binom{k}{j}c_3^j g_j(c_1)g_{k-j}(c_2),
\end{equation}
where $g_k(x)$ denotes the $k$th Jensen polynomial associated with the real entire function $\varphi$.
In particular, if $t=s=\frac{1}{2}$, then
\begin{eqnarray*}
 C_k^{\varphi, \varphi}\left( \frac{1}{2},\frac{1}{2}\right)&=&\left(\frac{1}{2} \right)^k \sum_{j=0}^k \binom{k}{j} 
 g_j(1) g_{k-j}(1).
\end{eqnarray*}
For the readers convenience, we give here (without proof) a teaser of choices for $\varphi$ in equation (\ref{examplephi}), and the resulting family of multiplier sequences.
\begin{enumerate}
\item If we set $\displaystyle{\varphi(x)=\sum_{k=0}^{\infty} \frac{x^k}{k! k!}}$,  and select $s,t \in (0,1)$, the multiplier sequence we obtain is
\[
\displaystyle{\seq{C_k^{\varphi, \varphi}(t,s)}=\seq{(1-s)^k \sum_{j=0}^k\left(\frac{1-t}{1-s} \right)^j \binom{k}{j} L_j\left(\frac{t}{t-1} \right)L_{k-j}\left( \frac{s}{s-1}\right)}},
\]
where $L_j(x)$ denotes the $j$th Laguerre polynomial (see for example \cite[Ch.\,12]{rainville}).
\item The choice $\displaystyle{\varphi(x)=\sum_{k=0}^{\infty} \frac{x^k}{(2k)!}}$, and $s,t \in (0,1)$ leads to the following multiplier sequences involving hypergeometric functions:
\[
\displaystyle{\seq{C_k^{\varphi, \varphi}(t,s)}=\seq{(1-s)^k \sum_{j=0}^k \binom{k}{j}\left(\frac{1-t}{1-s} \right)^j {}_1F_1\left[-j;\frac{1}{2};\frac{t}{4(t-1)}\right] {}_1F_1\left[-(k-j);\frac{1}{2};\frac{s}{4(s-1)}\right] }}.
\] 
\item Finally, selecting $\varphi(x)=e^{rx}$, and $s,t \in (0,1)$ leads to the multiplier sequences
\[
\displaystyle{\seq{C_k^{\varphi, \varphi}(t,s)}=\seq{(2+(s+t)(r-1))^k}}.
\]
\end{enumerate}
\end{example}
\section{Scholia and Open problems}\label{openproblems}
In the applications of the theory of Bessel functions (or hypergeometric functions ${}_pF_q$) it is frequently important to determine the distribution of zeros of certain combinations of Bessel functions (or hypergeometric functions).  In this section, we propose some techniques involving multiplier sequences that may  shed light on several questions and intriguing problems that arose in the course of our analysis, but remain unsolved at this time. Generalities aside, we commence here with a concrete example to illustrate the ideas involved. 
 
 \begin{example}\label{thirtytwo} Let $p(x):=cx^2+a x +b$,  where $a, b, c \ge 0$. If $a=b=1$ and $c>0$, so that $p(x)=cx^2+x+1$, then
\[
 \sum_{k=0}^\infty \frac{p(k)}{k!}x^k=e^x(1+x)(1+cx)\in\LP^+,
 \]
 and whence $\{p(k)\}_{k=0}^\infty$ is a multiplier sequence although $p$ need not have any real zeros.  Also, we hasten to note  that the {\it Hadamard product}, $\{\frac{p(k)}{k!}\}_{k=0}^\infty$,   of the two multiplier sequences $\{p(k)\}_{k=0}^\infty$ and $\{\frac{1}{k!}\}_{k=0}^\infty$  is again a multiplier sequence. In particular, it follows that $ \sum_{k=0}^\infty \frac{p(k)}{k!k!}x^k\in\LP^+$. We next consider the following query. Given a fixed $c>0$, does the  entire function $f(x):= (1+cx)I_0(2\sqrt{x})+\sqrt{x}I_1(2\sqrt{x})$, where $I_p(x)$ denotes the  modified Bessel function of the first kind of order $p$ (see equation (\ref{I_p})), have only real (positive) zeros? Calculating the Taylor coefficients of the entire function $f(x)$ (of order $1/2$), we find that 
 \[
 f(x)=\sum_{k=0}^\infty \frac{p(k)}{k!k!}x^k\in\LP^+
 \]
 since $\{\frac{p(k)}{k!}\}_{k=0}^\infty$ is a multiplier sequence  for any $c>0$.
If $p(x)$ is a quadratic polynomial with non-negative Taylor coefficients, then as was noted in the Introduction, $\{p(k)\}_{k=0}^\infty$ need not be a  multiplier sequence.  However, it is a noteworthy  fact that $\{\frac{p(k)}{k!}\}_{k=0}^\infty$ is always multiplier sequence, and $p(x):=cx^2+a x +b$,  where $a, b, c \ge 0$.
\end{example}
 \begin{proposition}\label{quad_by_fact}
 If $p(x):=cx^2+a x +b$,  where $a, b, c \ge 0$, then $\{\frac{p(k)}{k!}\}_{k=0}^\infty$ is a multiplier sequence.
\end{proposition}
\begin{proof} (An outline.) If $p(x)$ has only real zeros, the conclusion is clear, since $\{\frac{p(k)}{k!}\}_{k=0}^\infty$ is the Hadamard product of two multiplier sequences.  If $a\not= 0$ (so that $p(x):=cx^2+a x +b$, where we may assume that $b\not= 0$), then it suffices to consider $\frac{1}{b}p(bx/a)=1+x+\frac{bc}{a^2}x^2$. Hence we infer, from the argument used in Example \ref{thirtytwo}, that  $\{\frac{p(k)}{k!}\}_{k=0}^\infty$ is a multiplier sequence.
\end{proof}
The foregoing simple, but instructive, examples were introduced in order  to motivate the following general problem.
\begin{problem}\label{thirtyfour}   Let $p(x):=\sum_{k=0}^na_k x^k$, where $a_k \ge 0$ ($k\ge 0$).  Find conditions on $p(x)$ such that (a) $\{\frac{p(k)}{k!}\}_{k=0}^\infty$ is a multiplier sequence and (b) $\{\frac{p(k)}{k!}\}_{k=0}^\infty$ is a CZDS (cf. Definition \ref{CZDS}).
 \end{problem}
 Part (b) of Problem \ref{thirtyfour} appears to be particularly challenging. If $p(x)\in\LP^+\cap\mathbb{R}[x]$, then it follows from a theorem of Laguerre (see, for example, \cite[Theorem 4.1 (3)]{ccczds}), that the sequence $\{\frac{p(k)}{k!}\}_{k=0}^{\infty}$ is a CZDS. We also call attention to one  of the principal results of \cite[Theorem 2.13]{ccczdsb} which completely characterizes the {\it class of all polynomials} which interpolate CZDS. Notwithstanding, these results, at this juncture, we are obliged to expose our ignorance and formulate the following tantalizing open problem.
 \begin{problem} Is the sequence  $\{\frac{1+k+k^2}{k!}\}_{k=0}^\infty$ a CZDS?
 \end{problem}
  We next consider more complicated sequences involving the square root  function (see Section \ref{rootkoverkfact}).
 \begin{example} Let $\displaystyle{\alpha_k:=\frac{e^{-\sqrt{k}}}{k!}}$ and $\displaystyle{\beta_k:=\frac{e^{\sqrt{k}}}{k!}}$   for $k=0,1,2 \dots .$ Then the sequence $\alpha:=\{\alpha_k\}_{k=0}^\infty$  is not a multiplier sequence since the Jensen polynomial 
   \[
  g_3(x):= g_3(x;\alpha):=\sum_{k=0}^3 \binom{3}{k}\alpha_k x^k=1+\frac{3}{e}x+\frac{3}{2} e^{-\sqrt{2}} x^2+\frac{1}{6} e^{-\sqrt{3}} x^3
   \]
   has two non-real zeros (see the discussion after Definition \ref{jensenpoly}). In particular, $\psi(x):=\sum_{k=0}^{\infty}\frac{\alpha_k}{k!}x^k=\sum _{k=0}^{\infty } \frac{e^{-\sqrt{k}} x^k}{(k!)^2}\notin \LP^+$. On the other hand, our numerical work shows that the Jensen polynomials of degree $n$ ($1\le n\le 30$), associated with the sequence $\{\beta_k\}_{k=0}^\infty$, have only real zeros.
 \end{example}
   \begin{problem} \label{thirtysix} Determine whether the sequence  $\{\frac{\sqrt{k}}{k!}\}$ is (a) a multiplier sequence  (b) a CZDS (cf. Definition \ref{CZDS}).
\end{problem}
We remark that if the sequence  $\seq{\frac{\sqrt{k}}{k!}}$ is  a multiplier sequence (or a CZDS), then the sequence  $\{\frac{\sqrt{k}}{(2k)!}\}_{k=0}^{\infty}$ is  also a multiplier sequence (or a CZDS). Indeed, it follows from the Legendre Duplication Formula \cite[p.\,71]{andrews} that $\ds{\frac{\sqrt{\pi}}{4^k\Gamma(k+1/2)}=\frac{k!}{(2k)!}}$. Now by Laguerre's theorem (\cite[Theorem 4.1 (3)]{ccczds}),  the sequence $\ds{\{\frac{1}{\Gamma(k+1/2)}\}_{k=0}^{\infty}}$ is a CZDS and whence the above claim follows.\\
  In light of the discussion in Example \ref{thirtytwo} and the fact that the sequence $\ds{\{\frac{1}{\Gamma(k+1)}\}_{k=0}^{\infty}}$ is a CZDS, we expect an affirmative answer to the following question (see also Theorem 1).
  \begin{problem} Is it true that for every $s \in \mathbb{R}^+$, there exists an $m \in \mathbb{N}$, such that $\displaystyle{\sum_{k=0}^{\infty} \frac{k^s}{(k!)^m}x^k \in\LP^+}$.
    \end{problem}
Before stating our next  problem, we pause for a moment  and briefly touch upon the characterization of entire functions in $\varphi(x)
\in \LP^+$  in terms of their Taylor coefficients. To this end, we
consider
the entire function
\begin{equation} \label{genentire}
\varphi(x):=\sum_{k=0}^{\infty}\alpha_kx^k,\quad \text{where}\quad
\alpha_k=\frac{\gamma_k}{k!}, \quad \gamma_0=1,\quad \gamma_k\ge0\quad
(k=1,2,3\dots).
\end{equation}
and recall the following definition.
\begin{definition} A real sequence $\{\alpha_k\}_{k=0}^{\infty}$,
 $\alpha_0=1$, is said to be a {\it totally positive sequence}, if the infinite
lower triangular matrix
\[
A=(\alpha_{i-j})=\left(\begin{array}{cccccc} \alpha_0 & 0 & 0 & 0 & 0 & \cdots\\
\alpha_1&\alpha_0&0&0&0&\cdots\\
\alpha_2&\alpha_1&\alpha_0&0&0&\cdots\\
\alpha_3&\alpha_2&\alpha_1&\alpha_0&0&\cdots\\
\vdots &\vdots &\vdots & \vdots &\vdots & \ddots \end{array} \right) \qquad (i,j=1,2,3,\dots), 
\]
is totally positive; that is, all the minors of $A$ of all orders are
non-negative.
\end{definition}
In \cite{AESW}, M. Aissen, A. Edrei, I. J. Schoenberg and A. Whitney
characterized the generating functions of totally positive sequences. A special
case of their result is the following theorem.
 \begin{theorem}(\cite[p.\,306]{AESW})\label{thirtynine} Let $\varphi(x)$ be the 
entire
function defined by $(\ref{genentire})$. Then $\{\alpha_k\}_{k=0}^{\infty}$ is a totally
positive sequence if and only if $\varphi(x) \in \LP^+$.
\end{theorem}
Preliminaries aside, we are now in position to state and analyze the next open problem.
\begin{problem}\label{fourty} Let $\ds{\gamma_k:= \frac{1}{(k+1)^{k+1}}}$, $k\in \mathbb{N}_0$. Determine whether the sequence $\seq{\frac{\gamma_k}{k!}}$ is (a) a multiplier sequence (b) a CZDS.
\end{problem}
We first claim that the entire function $\psi(x):=\sum_{k=0}^\infty \gamma_k x^k\notin \LP^+$. In order to verify that $\psi(x)$ is not in the Laguerre-P\'olya class, we invoke Theorem \ref{thirtynine} and show that  $\{\gamma_k\}_{k=0}^{\infty}$ is \underbar{not} a totally
positive sequence.  Indeed, after some experimentation, we find that the determinant of the $4\times 4$ submatrix 
\[
A=\left( \begin{array}{cccc} \frac{1}{4} & 1 & 0 & 0 \\ \frac{1}{27} &  \frac{1}{4} & 1& 0 \\ \frac{1}{256} & \frac{1}{27}& \frac{1}{4}& 1 \\ \frac{1}{3125} & \frac{1}{256} & \frac{1}{27} & \frac{1}{4} \end{array} \right)
\]
  is negative: $\det(A)=-(38873/1166400000)$.   We remark that {\bf if} $\seq{\frac{\gamma_k}{k!}}$ is  a multiplier sequence, then $\varphi(x):=\sum_{k=0}^{\infty} \frac{\gamma_k}{k!}x^k\in\LP^+$ and therefore, by Theorem \ref{thirtynine},  the sequence $\seq{\frac{\gamma_k}{k!}}=\seq{\frac{1}{(k+1)^{k+1}}\frac{1}{k!}}$ is a totally positive sequence. In addition, we also observe that by Stirling's formula (\cite[p.\,98]{andrews})
  \[
  \frac{1}{(k+1)^{k+1}}\sim \frac{1}{(k+1)!}\, \sqrt{2\pi}\, e^{-(k+1)}\, \sqrt{k+1} \qquad (k\gg1).
\]
Thus noting again the vexing presence of the square root function, it may be instructive to compare Problem \ref{fourty} with Problem \ref{thirtysix}.
\medskip
We conclude this paper  with one more open problem that (i) may be useful in the study of CZDS and (ii) is related to our results in Section \ref{transformations}.

\begin{problem} Characterize all multiplier sequences which have a $C_k$-representation (cf. Definition \ref{c_krep}).
\end{problem}

\end{document}